\documentclass[11pt,a4paper]{amsart}
\usepackage{graphicx}
\usepackage{color,xcolor}
\usepackage{a4wide}
\usepackage{amsmath, amssymb}
\usepackage{amsthm}

\newtheorem{theorem}{Theorem}
\newtheorem{proposition}[theorem]{Proposition}%
\newtheorem{corollary}[theorem]{Corollary}
\newtheorem{lemma}[theorem]{Lemma}
\newtheorem{remark}[theorem]{Remark}%

\def\rin{\rotatebox[origin=c]{90}{$\in$}}

\begin{document}

\title[Geometric properties of Blaschke-like maps]
      {Geometric properties of Blaschke-like maps
       on domains with a conic boundary}

\author{Masayo Fujimura}
\address{Department of Mathematics, National Defense Academy of Japan, Japan}
\email{masayo@nda.ac.jp}

\author{Yasuhiro Gotoh}
\address{Department of Mathematics, National Defense Academy of Japan, Japan}
\email{gotoh@nda.ac.jp}

\begin{abstract}
For a circle $ C $ contained in the unit disk,
the necessary and sufficient condition for the existence of a triangle
inscribed in the unit circle 
and circumscribed about $ C $ is known as Chapple's formula.
The geometric properties of Blaschke products of degree 3
given by Daepp et al. (2002) and Frantz (2004)
allow us to extend Chapple's formula to the case of 
ellipses in the unit disk.
The main aim of this paper is to provide a further extension 
of Chapple's formula.
Introducing a Blaschke-like map of a domain whose boundary 
is a conic, we extend their results to the case where
the outer curve is an ellipse or a parabola. 
Moreover, we also give some geometrical properties for
the Blaschke-like maps of degree $ d $.
\end{abstract}


\maketitle


\section{Introduction}
For a circle $ C:  \vert z-c \vert =r $ contained in  
the unit disk $ \mathbb{D} $, 
there exists a triangle inscribed in the unit circle 
and circumscribed about $ C $ if and only if $ C $ satisfies 
$ \vert c \vert^2=1-2r $.
This result was independently 
given by Chapple \cite{chapple} and Euler (1765), 
so it is called Chapple's formula or the Chapple-Euler formula.

The geometrical properties of the Blaschke product allow us to 
extend the inner circle $ C $ of this formula to an ellipse
\cite{daepp}, \cite{frantz}
(cf. \cite{fuji-cmft} for an extension of Fuss' formula \cite{fuss}
 in the case of quadrilaterals).

A Blaschke product of degree $ d $ is a rational map defined by
\begin{equation}\label{eqdef:bla}
  B(z)=e^{i\theta}\prod_{k=1}^{d}\frac{z-a_k}{1-\overline{a_k}z}\quad
    (a_k\in\mathbb{D},\ \theta\in\mathbb{R}).
\end{equation}
In the case that $ \theta=0 $ and $ B(0)=0 $, $ B $ is called {\it canonical}.
It is enough to consider only the canonical Blaschke products 
to study the geometrical properties of the inverse images 
of the Blaschke products.
In fact, for 
\[
    f_1(z)=e^{-\frac{\theta}d i}z \quad \mbox{and}\quad
    f_2(z)=\frac{z-(-1)^da_1\cdots a_d e^{i\theta}} 
                {1-(-1)^d\overline{a_1\cdots a_d e^{i\theta}}z} 
\]
the composition $ f_2\circ B\circ f_1 $ is canonical, and
the geometrical properties of the preimages of these two Blaschke products
$ B $ and $ f_2\circ B \circ f_1 $ are the same.
We remark that, since the derivative of a Blaschke product has no zeros on
$ \partial\mathbb{D} $ (see, for instance \cite[Lemma 3.1]{Inner}),
there are $ d $ distinct preimages of $ \lambda\in\partial\mathbb{D} $ 
by $ B $.

Let $ w_1,\cdots,w_d $ be the $ d $ distinct preimages of 
$ \lambda\in\partial\mathbb{D} $ by $ B $
and $ \ell_{\lambda} $ the set of lines joining $ w_j $ and $ w_k $ ($j\neq k $).
Then, the envelope $ I_B $ of the family of lines 
$ \{\ell_{\lambda}\}_{\lambda\in\partial\mathbb{D}} $ is called 
the {\it interior curve associated with} $ B $.
The interior curve associated with a Blaschke product of degree $ 3 $
forms an ellipse.

\begin{theorem}[Daepp, Gorkin, and Mortini {\cite[Theorem 1]{daepp}}]
\label{thm:1}
 Let $ B $ be a Blaschke product of degree $ 3 $ with zeros at
 the points $ 0,a_1$, and $a_2 $.
 For $ \lambda\in\partial\mathbb{D} $, let
 $ w_1,w_2$, and $w_3 $ denote the points mapped to $ \lambda $ 
 under $ B $. Then the lines joining $ w_j $ and $ w_k $ for $ j\neq k $
 are tangent to the ellipse $ E $ with equation
 $  \vert w-a_1 \vert + \vert w-a_2 \vert =\vert 1-\overline{a_1}a_2\vert $.
\end{theorem}

The above result is related to the following Poncelet's theorem \cite{pon1},
known as Poncelet's porism.
For details of this theorem, see, for example, \cite{flatto}.

\begin{theorem}[Poncelet \cite{pon1}]\label{thm:PON}
Let $ C_1 $ and  $C_2 $ be two smooth conics.
Suppose there is an $ n$-sided polygon inscribed in  $C_1 $ 
and circumscribed about $ C_2 $.
Then for any point $ p $ of $ C_1 $, there exists an  $n$-sided polygon
with $ p $ as a vertex, inscribed in $ C_1 $  and circumscribed about $ C_2 $.
\end{theorem}

The $n$-sided polygons inscribed in $ C_1$  and circumscribed about
$ C_2 $ are called the
{\it Poncelet $n$-polygons with respect to $ C_1 $ and $ C_2 $}
(in this paper, we do not consider non-convex polygons such as stellations).
Also, $ C_2 $ is called an {\it $ n $-inscribed conic in $ C_1 $}.

Theorem \ref{thm:1} shows that each Blaschke product of degree 3 constructs 
a $ 3 $-inscribed ellipse in $ \partial\mathbb{D} $.
The question arises whether every $ 3 $-inscribed ellipse in 
$ \partial\mathbb{D} $ can be
 constructed from some Blaschke product.
Frantz \cite{frantz} gave the answer to this question.

\begin{theorem}[Frantz {\cite[Proposition 3]{frantz}}]\label{thm:Fra}
An ellipse $ C $ is a $3$-inscribed ellipse in $\partial\mathbb{D}$ 
if and only if $ C $ is the interior curve with respect to a Blaschke product
of degree $ 3 $.
\end{theorem}

Thus there is a close relationship between the interior curves associated 
with Blaschke products and $ 3 $-inscribed ellipses in 
$ \partial\mathbb{D} $.
If we could extend these results to the case that the outer curve is a conic,
we would find results more closely related to Poncelet's theorem.

In this paper, we introduce a Blaschke-like map on a domain
$ D $ whose boundary is an ellipse or a parabola by using
a conformal map from the unit disk to $ D $.
We then study the geometric properties of the Blaschke-like maps.

Let $ \varphi_t $ be the following Joukowski transformation
\begin{equation}\label{eq:jow}
 z=\varphi_t(w)=\frac{1}{1+t^2}\Big(t^2w+\dfrac{1}{w}\Big) \qquad (0<t<1).
\end{equation}
The transformation $ \varphi_t $ conformally maps 
the unit disk $ \mathbb{D}$ in the $ w $-plane onto 
the exterior of the elliptical disk $ \mathbb{E}_t $
with semi-minor axis $(1-t^2)/(1+t^2) $ and semi-major axis $ 1 $.
Let $ E_t=\partial \mathbb{E}_t  $.
We remark that each ellipse is similar to the ellipse given by $ E_t$ for some
$ t $ with $ 0<t<1 $.
Therefore, in the following discussion,
it is sufficient to consider ellipses of the form $ E_t $.

For a canonical Blaschke product $ B $, set 
$ B_{\varphi_t}=\varphi_t\circ B\circ \varphi_t^{-1} $,
$ \varphi_t $ is the conformal map mentioned above.
We call $ B_{\varphi_t} $ a {\it Blaschke-like map associated with 
$ B $ and $ \varphi_t $}.


We will show the following result, which is an extension of Theorem
\ref{thm:1} (see Figure \ref{pic:daen}
     \footnote{Figures \ref{pic:daen} and \ref{pic:parab} are drawn by
      using GeoGebra ({\tt https://www.geogebra.org/)}}).

\begin{theorem}\label{thm:daen}
  Let $ B_{\varphi_t} $ be a Blaschke-like map associated with
  a Blaschke product $ B $ of degree $ 3 $ and $ \varphi_t $.
  Then, the interior curve with respect to $ B_{\varphi_t} $
  is an ellipse.
\end{theorem}

\begin{figure}[htbp]
\centerline{
  \fbox{\includegraphics[width=0.4\linewidth]{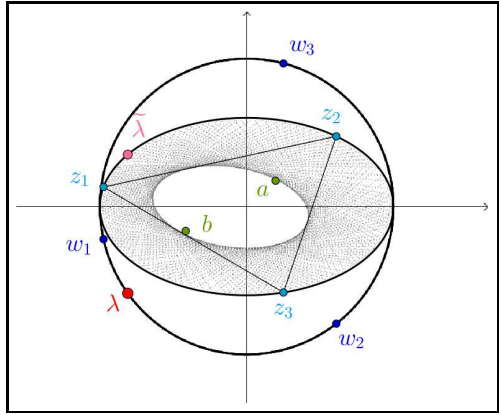}}}
\caption{The interior curve with respect to $ B_{\varphi_t} $ constructed 
         from $ B(w)=w(w-a)(w-b)/((1-\overline{a}w)(1-\overline{b}w)) $ with
         $ a=0.2+0.17i,b=-0.42-0.17i $ and $ t=0.5 $.}
\label{pic:daen}
\end{figure}

The following result gives an extension of Theorem \ref{thm:Fra}.

\begin{corollary}\label{cor:unique}
  For each ellipse $ E_t\ (0<t<1) $,
  $ C_2 $ is a $3$-inscribed ellipse in $ E_t $
  if and only if $ C_2 $ is the interior curve with respect to
  a Blaschke-like map $ B_{\varphi_t} $ for some Blaschke product $ B $ of
  degree $ 3 $.
\end{corollary}

R. Schwartz and S. Tabachnikov study the loci of the center of mass in
several meanings of Poncelet $ n $-sided polygons \cite{taba}.
In particular, they show the following result for the barycenter of mass
for the vertices of polygons.

\begin{theorem}[Schwartz and Tabachnikov {\cite[Theorem 1]{taba}}]
Let $ C_2 \subset C_1 $  be a pair of nested ellipses that
admit a $1$-parameter family of Poncelet $ n $-sided polygons  $\mathcal{P}_t $.
Then the locus of centers of mass for the vertices of each  $ \mathcal{P}_t $
is an ellipse similar to $ C_1 $ or a single point.
\end{theorem}

For $ d>3 $, the interior curve with respect to a Blaschke product of degree
$ d $ is not always an ellipse.
However, the locus of the center of mass for the preimages
$ B_{\varphi_t}^{-1}(\widetilde{\lambda}) $ of each 
$ \widetilde{\lambda}\in\mathbb{E}_t $ forms an ellipse.

\begin{proposition}\label{thm:jushin}
Let $ z_1,\cdots,z_d $ be the $ d $ distinct preimages of
$ \widetilde{\lambda}\in E_t $ by $ B_{\varphi_t} $.
As $ \widetilde{\lambda} $ ranges over $ E_t $,
the center of mass $ w=(z_1+\cdots+z_d)/d $ 
of $ n $-sided polygon with vertices $ z_1,\cdots,z_d $,
forms an ellipse which is similar to $ E_t $ or a single point.
\end{proposition}

We treat the case of the outer curve being a parabola in the same way.
Let
\[
  z=\psi_t(w)=\Big(\frac{1-w}{1+w}+t\Big)^2-t^2 \quad (t>0).
\]
The transformation $ \psi_t $ conformally maps the unit disk in the 
$ w $-plane onto the domain 
$ \mathbb{P}_t=\{ z\in\mathbb{C}\,;\,z=x+iy,\ y^2+4t^2x>0 \} $.
Set $ P_t=\partial \mathbb{P}_t $.

We remark that each parabola is similar to the parabola given by $ P_t$ 
for some $ t $ with $ t>0 $.
Therefore, in the following discussion,
it is sufficient to consider the parabola of the form $ P_t $.

For $ \psi_t $ and a canonical Blaschke product $ B $, we 
define the Blaschke-like map  $  B_{\psi_t}=\psi_t\circ B\circ \psi_t^{-1} $
as in the case of Joukowski transformation
(see Section \ref{sec:parab} for details).
For $ B_{\psi_t} $, we have the following results 
(see also Figure \ref{pic:parab}).

\begin{theorem}\label{thm:para}
  Let $ B_{\psi_t} $ be a Blaschke-like map associated with a Blaschke product
  $ B $ of degree $ 3 $ and $ \psi_t $.
  Then, the interior curve with respect to $ B_{\psi_t} $
  is an ellipse.
\end{theorem}

\begin{figure}[htbp]
\centerline{
  \fbox{\includegraphics[width=0.6\linewidth]{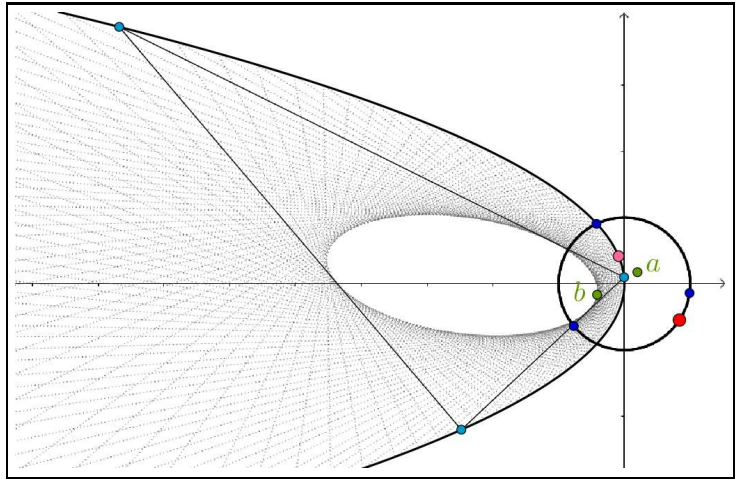}}}
\caption{The interior curve with respect to $ B_{\psi_t} $ constructed 
         from $ B(w)=w(w-a)(w-b)/((1-\overline{a}w)(1-\overline{b}w))$ with
         $ a=0.2+0.17i,b=-0.42-0.17i $ and $ t=0.7 $.}
\label{pic:parab}
\end{figure}

\begin{corollary}\label{cor:para}
  For each parabola $ P_t \ (t>0)$,
  $ C_2 $ is a $3$-inscribed ellipse in  $ P_t $
  if and only if $ C_2 $ is the interior curve with respect to
  a Blaschke-like map $ B_{\psi_t} $ for some Blaschke product $ B $ of
  degree $ 3 $.
\end{corollary}

Does the conformal transformation that maps 
a disk onto a domain whose boundary is a conic always
induce an elliptic interior curve?
The answer to this question is ``no''.
We will give an example in Section \ref{sec:inside}.

\section{Outside of elliptic disks}
\subsection{The standard forms and the general forms}\

An ellipse on the complex plane is represented by the following 
two different forms of equations.
\begin{align}
  \label{eq:standard}
   & \vert z-f_1\vert+\vert z-f_2 \vert 
     =r \quad(f_1\neq f_2,\ r> \vert f_1-f_2 \vert ) \\
  \label{eq:general}
   & \overline{u}z^2+pz\overline{z}+u\overline{z}^2+\overline{v}z
     +v\overline{z}+q=0 \quad(p,q\in\mathbb{R},\ p^2-4u\overline{u}>0)
\end{align}
Here, we assume that the coefficients in equation \eqref{eq:general} 
also satisfy the condition
\begin{equation}\label{eq:add-condition}
  p\big(pv\overline{v}-u\overline{v}^2-\overline{u}v^2
                           +q(4u\overline{u}-p^2)\big)>0 
\end{equation}
that the ellipse does not degenerate to a point or the empty set.

An equation of the form \eqref{eq:standard} is called the 
{\it standard form} of an ellipse,
\eqref{eq:general} is called the {\it general form}.
In this paper, both forms are used as needed.
In \eqref{eq:standard}, $ f_1$ and $ f_2 $ are the foci of the ellipse,
and $ r $ is the sum of distances from the two foci to a point on the ellipse.
Note that equations \eqref{eq:standard} and \eqref{eq:general} for
$ f_1=f_2 $ and $ u=0 $, respectively, are equations of a circle.

Expanding  equation \eqref{eq:standard},
we obtain the equation in the following  general form.
\begin{align}\notag
  & (\overline{f_1}-\overline{f_2})^2z^2
    +2((f_1-f_2)(\overline{f_1}-\overline{f_2})-2r^2)z\overline{z}
    +(f_1-f_2)^2\overline{z}^2\\ \notag
 & \quad
   -2((\overline{f_1}-\overline{f_2})(f_1\overline{f_1}-f_2\overline{f_2})
      -r^2(\overline{f_1}+\overline{f_2}))z\\ \notag
 & \quad 
  -2((f_1-f_2)(f_1\overline{f_1}-f_2\overline{f_2})-r^2(f_1+f_2))\overline{z}\\
    \label{eq:StoG}
 &\quad
    +(f_1\overline{f_1}-f_2\overline{f_2})^2 
   -2(f_1\overline{f_1}+f_2\overline{f_2})r^2+r^4=0.
\end{align}
Conversely, the following gives the transformation 
from general form to standard form.

\begin{lemma}\label{lemma:GtoS}
 The foci $ f_1 $ and $ f_2 $  of the ellipse \eqref{eq:general} are
 the solutions of the following equation
  \begin{equation}\label{eq:f1f2}
    (4u\overline{u}-p^2)\zeta^2+(4u\overline{v}-2pv)\zeta+4qu-v^2=0,
  \end{equation}
 and the sum $ r $ is given by 
  \begin{equation}\label{eq:r}
     r=\frac12 \vert f_1-f_2 \vert\sqrt{2+\Big\vert\frac{p}{u}\Big\vert}.
  \end{equation}
\end{lemma}

\begin{proof}
As the ellipse given by  equation \eqref{eq:general}
satisfies the condition $ p^2-4 \vert u \vert^2>0 $, we have $ p\neq0 $.
Therefore, equations \eqref{eq:general} and \eqref{eq:StoG} both have a
non-zero coefficient of $ z\overline{z} $.
If these are equations for the same ellipse, then by eliminating 
$ z\overline{z} $, the following identity holds
\[
  p\times\big(\mbox{the left side of Eq.}\eqref{eq:StoG}\big)
  -2( \vert f_1-f_2 \vert^2-2r^2)\times
    \big(\mbox{the left side of Eq.}\eqref{eq:general}\big)\equiv 0. 
\]
Comparing each coefficient, we have
\begin{align}
   \label{eq:a}
  &  2(2r^2- \vert f_1-f_2 \vert^2)u+(f_1-f_2)^2p=0,\\
   \label{eq:i}
  & (2r^2- \vert f_1-f_2 \vert^2)v
     -\big((f_1-f_2)(\vert f_1\vert^2-\vert f_2\vert^2)-r^2(f_1+f_2)\big)p=0, 
    \mbox{\quad and}\\
   \label{eq:u}
  &  2(2r^2-\vert f_1-f_2 \vert^2)q+\big((\vert f_1\vert^2-\vert f_2\vert^2)^2
     -2r^2(\vert f_1\vert^2+\vert f_2\vert^2)+r^4\big)p=0.
\end{align}
Eliminating $ r^2 $ from \eqref{eq:a} and \eqref{eq:i} gives
\[
    p(f_1-f_2)^2\big(2(\overline{f_1}+\overline{f_2})u+(f_1+f_2)p+2{v}\big)=0.
\]
Since $ f_1\neq f_2 $ and $ p\neq 0 $, the following holds
\[
    2(\overline{f_1}+\overline{f_2})u+(f_1+f_2)p+2{v}=0.
\]
Eliminating $ \overline{f_1}+\overline{f_2} $ from the above equation
and the complex conjugate of this equation, we have
\begin{equation}\label{eq:wa}
  f_1+f_2=\frac{2pv-4u\overline{v}}{4u\overline{u}-p^2}.
\end{equation}

Since $ p\neq 0 $, eliminating $ r^2 $ from \eqref{eq:a} 
and its complex conjugate \eqref{eq:a} gives the following
\[
  \big((\overline{f_1}+\overline{f_2})^2-4\overline{f_1}\overline{f_2}\big)u
  -\big((f_1+f_2)^2-4f_1f_2\big)\overline{u}=0,
\]
and then substituting \eqref{eq:wa} into the above equation
and multiplying by $ -(4u\overline{u}-p^2)/4 $, 
we obtain the following
\begin{equation}\label{sekisekib}
   -\overline{u}(4u\overline{u}-p^2)f_1f_2
     +u(4u\overline{u}-p^2)\overline{f_1}\overline{f_2}
    +\overline{v}^2u-\overline{u}v^2=0.
\end{equation}
Moreover, eliminating  $ r^2 $ from 
\eqref{eq:a} and \eqref{eq:u} gives
\begin{align*}
 & \big(4(f_1+f_2)^2(\overline{f_1}+\overline{f_2})^2
    -64f_1\overline{f_1}f_2\overline{f_2}\big)u^2\\
 & +\Big(4(f_1+f_2)(\overline{f_1}+\overline{f_2})\big((f_1+f_2)^2-4f_1f_2\big)p
   -16\big((f_1+f_2)^2-4f_1f_2\big)q\Big)u\\
 &+\big((f_1+f_2)^2-4f_1f_2\big)^2p^2=0,
\end{align*}
and substituting \eqref{eq:wa} into the above equation
and multiplying by $-(4u\overline{u}-p^2)^2/16$, we have
\begin{align*}
  & -p^2(4u\overline{u}-p^2)^2f_1^2f_2^2
   +4u^2(4u\overline{u}-p^2)^2f_1\overline{f_1}f_2\overline{f_2}\\
  & \quad 
    -2(4u\overline{u}-p^2)\big(8q\overline{u}u^2
    -2(p\overline{v}v+qp^2)u+p^2v^2\big)f_1f_2
    +(4qu-v^2)(2\overline{v}u-pv)^2=0.
\end{align*}
Using \eqref{sekisekib}, eliminating $ \overline{f_1}\overline{f_2} $ 
from the above equation, we have
\begin{equation}\label{eq:seki}
    \big((4u\overline{u}-p^2)f_1f_2-4qu+v^2\big)
    \big((4u\overline{u}-p^2)^2f_1f_2-(2u\overline{v}-pv)^2\big)=0.
\end{equation}
If the second factor of \eqref{eq:seki} is $ 0 $,
the equation having $ f_1 $  and $ f_2 $ as its
solutions is given by 
\[
    \zeta^2-(f_1+f_2)\zeta+f_1f_2
     =\Big(\zeta-\frac{2u\overline{v}-pv}{p^2-4u\overline{u}}\Big)^2=0,
\]
and the two solutions satisfy 
$ f_1=f_2=(2u\overline{v}-pv)/(p^2-4u\overline{u}) $.
This contradicts $ f_1\neq f_2 $.
Therefore, the first factor of \eqref{eq:seki} must be $ 0 $,
and the solutions of the equation
\[
  (4u\overline{u}-p^2)\big(\zeta^2-(f_1+f_2)\zeta+f_1f_2\big)
  =(4u\overline{u}-p^2)\zeta^2+(4u\overline{v}-2pv)\zeta+4qu-v^2=0
\]
give the two foci.

From \eqref{eq:a}, $ r $ is written as 
$ r^2=\vert f_1-f_2\vert^2/2-p(f_1-f_2)^2/(4u) $.
Since $ r^2>\vert f_1-f_2\vert^2 $ and $ r\in\mathbb{R} $,
$ -p(f_1-f_2)^2/(4u)>0 $ holds.
Therefore, we have 
\begin{align*}
 r^2 &=\bigg\vert \frac12 \vert f_1-f_2\vert^2
              -\frac{p}{4u}(f_1-f_2)^2\bigg\vert
              =\frac12 \vert f_1-f_2\vert^2
        +\Big\vert\frac{p}{4u}\Big\vert\vert f_1-f_2\vert^2
      =\frac14\Big(2+\Big\vert\frac{p}{4u}\Big\vert\Big)\vert f_1-f_2\vert^2,
\end{align*}
and $ r $ is given by \eqref{eq:r}.
\end{proof}

\begin{remark}
Equation \eqref{eq:f1f2} has two solutions counting multiplicity,
because \eqref{eq:f1f2} has a non-zero leading coefficient.

In the case of $ f_1=f_2 $,
the corresponding ellipse is a circle, and 
equation \eqref{eq:standard} can be written as 
\[
     \vert z-f_1 \vert =\frac{r}2 \quad (r>0). 
\]
The general form of a circle is
\[
    pz\overline{z}+\overline{v}z+v\overline{z}+q=0\quad
    (p,q\in\mathbb{R}, \ p\neq 0,\ v\overline{v}-pq>0) ,
\]
obtained by substituting $ u=0 $ in equation \eqref{eq:general}.
In addition, even when $ f_1=f_2 $, equation \eqref{eq:StoG} holds 
and equation \eqref{eq:f1f2} has
the multiple solution $ \zeta=-v/p $.
But, the right-side of equation \eqref{eq:r} is identically zero,
and the value $ r $ cannot be determined from \eqref{eq:r}.
In this case, $ r $ is obtained as a
solution to $ r^2=4(v\overline{v}-pq)/p^2 $ from \eqref{eq:u}.
\end{remark}

\subsection{Conformal deformation from the unit disk to 
            an outside of elliptical disk}\label{sec:DGM}\

The Joukowski transformation 
\[
 z=\varphi_t(w)=\frac{1}{1+t^2}\Big(t^2w+\dfrac{1}{w}\Big) \qquad (0<t<1),
\]
conformally maps the unit disk $ \mathbb{D} $ in the $ w $-plane onto the
exterior of the elliptic disk $ \mathbb{E}_t $ in the $ z $-plane
whose semi-major and semi-minor axes
are $ 1$ and $(1-t^2)/(1+t^2) $, respectively. 
Set $ E_t=\partial \mathbb{E}_t  $.
Note that $ \varphi_t $ is a continuous map from $ \overline{\mathbb{D}} $ 
to $ \widehat{\mathbb{C}}\setminus\mathbb{E}_t $.
The foci of $ {E}_t $ are $ \pm (2t)/(1+t^2) $ and
the eccentricity is given by $  e=(2t)/(1+t^2)$. 
Therefore, for any eccentricity $ e $ ($0<e<1$),
the unit circle can be mapped to an ellipse of eccentricity
$ e $ by $ \varphi_t $ with suitable value of $ t $.

The ellipse $ {E}_t $ is represented as
\begin{align}\label{eq:E1}
       \Big\vert z-\frac{2t}{1+t^2}\Big\vert
          +\Big\vert z+\frac{2t}{1+t^2}\Big\vert=2
          \qquad (0<t<1), \mbox{\quad or} \\
        \label{eq:E2}
   t^2z^2-(1+t^4)z\overline{z}+t^2\overline{z}^2+(1-t^2)^2=0\qquad (0<t<1).
\end{align}
For a canonical Blaschke product $ B $, 
let $ B_{\varphi_t}=\varphi_t\circ B\circ \varphi_t^{-1} $ 
(see Figure \ref{pic:joukow}).

\begin{figure}[htbp]
    \centering
    \includegraphics[width=0.7\linewidth]{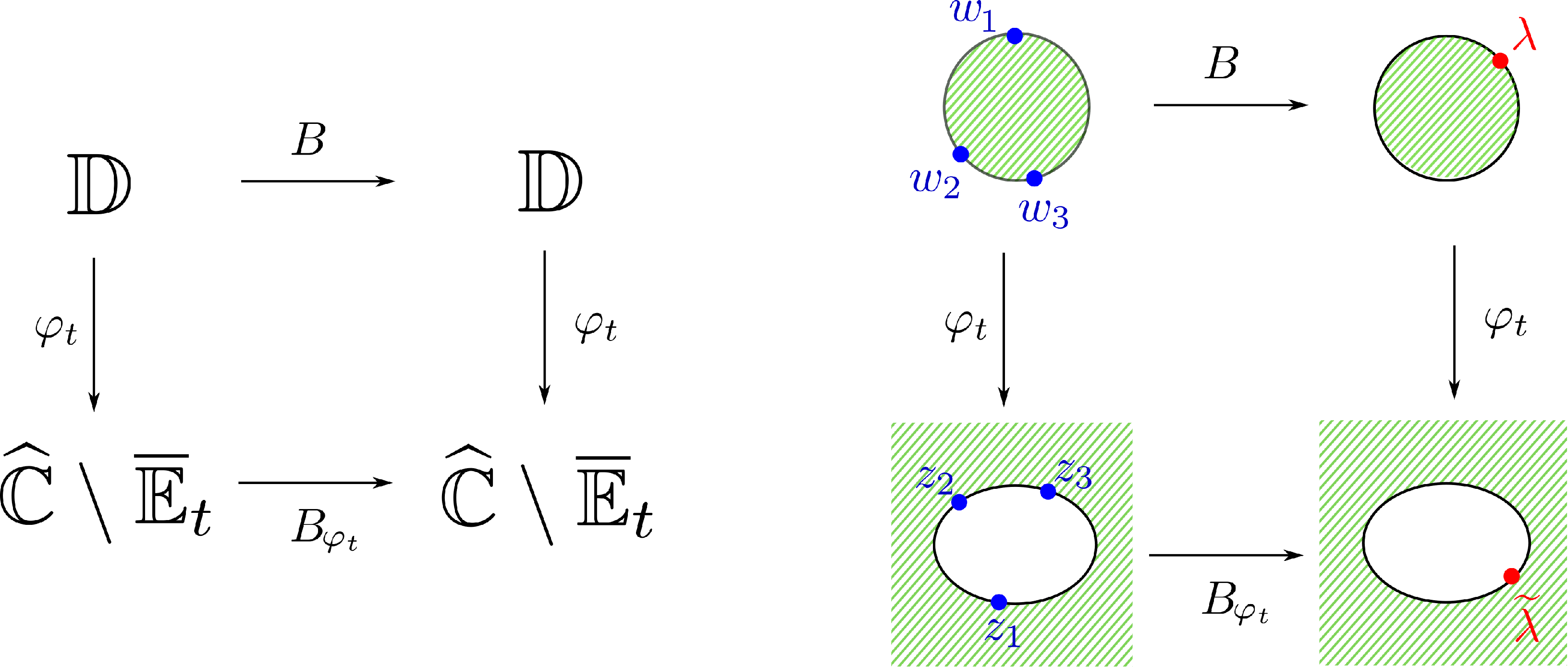}
    \caption{The Blaschke-like map
          $ B_{\varphi_t}=\varphi_t\circ B\circ\varphi_t^{-1} $.}
    \label{pic:joukow}
\end{figure}

Since $ \varphi_t(1/(t^2w))=\varphi_t(w) $,
the map $ \varphi_t $ conformally maps $ \{1/t^2<\vert w\vert \leq\infty\} $
onto $ \widehat{\mathbb C}\setminus\overline{\mathbb E}_t $ 
as well as conformally maps $ \mathbb{D} $ onto 
$ \widehat{\mathbb C}\setminus\overline{\mathbb E}_t $.
Therefore, for each $ z\in \widehat{\mathbb C}\setminus\overline{\mathbb E}_t $
we can choose a unique branch $ w $ of $ \varphi_t^{-1}(z) $
that satisfies $ \vert w\vert< 1 $.

Then, the map $B_{\varphi_t} $ maps the exterior of the elliptic disk 
$ \widehat{\mathbb{C}}\setminus\overline{\mathbb{E}}_t $ onto itself.
We call $ B_{\varphi_t} $ a {\it Blaschke-like map associated with 
$ B $ and $ \varphi_t $}.

\subsubsection{The interior curve of $ B_{\varphi_t} $}\

The Joukowski transformation $ z=\varphi_t(w) $ is a continuous 
map on $ \overline{\mathbb{D}} $, and on $ \partial\mathbb{D} $
it can be written 
\begin{equation}\label{map:phi}
    z=\varphi_t(w)=\frac{t^2w+\overline{w}}{1+t^2} \quad
       (w\in\partial\mathbb{D})
\end{equation}
since $ w\overline{w}=1 $ holds.
Therefore, on $ E_t $, the inverse can be expressed as 
\begin{equation}\label{map:phiIV}
      w=\varphi^{-1}_t(z)=\frac{\overline{z}-t^2z}{1-t^2}\qquad (z\in E_t). 
\end{equation}

Here, we prove Theorem \ref{thm:daen}.

\begin{proof}[Proof of Theorem \ref{thm:daen}]
  Let $ B $ be a canonical Blaschke product 
  $ B(w)=w(w-a)(w-b)/((1-\overline{a}w)(1-\overline{b}w)) $ of degree $ 3 $,
  and  $ B_{\varphi_t} $ the Blaschke-like map associated with
  $ B $ and $ \varphi_t $.
  For $ \widetilde{\lambda}\in E_t $, let $ z_1,\ z_2,\ z_3 $ 
  be the points mapped to $ \widetilde{\lambda} $ under $ B_{\varphi_t} $,
  and set
  \[
     \lambda=\frac{\overline{\widetilde{\lambda}}-t^2\widetilde{\lambda}}
                  {1-t^2},\qquad
     w_k=\frac{\overline{z_k}-t^2{z_k}}
                  {1-t^2} \quad (k=1,2,3).
  \]
  Then, it is clear that
  $ \widetilde{\lambda}=\varphi_t(\lambda), \ z_k=\varphi_t(w_k) $
  and $ B(w_k)=\lambda \ (k=1,2,3) $ hold for
  $ \lambda $ and $ w_k\ (k=1,2,3) $, where $ w_k\ (k=1,2,3) $ are
  the solutions of 
  \[
    w(w-a)(w-b)-\lambda(1-\overline{a}w)(1-\overline{b}w)=0.
  \]
  So, the following equalities hold from Vieta's formula
  \begin{equation}\label{eq:Bw}
      w_1+w_2+w_3=a+b+\lambda\overline{a}\overline{b}, \quad
      w_1w_2+w_1w_3+w_2w_3=ab+\lambda(\overline{a}+\overline{b}), \quad
      w_1w_2w_3=\lambda.     
  \end{equation}
  For each $ k,j=1,2,3 $ ($k<j$), the equation of line $ l_{kj} $ joining 
  two points $ z_k $ and $ z_j $ 
  is given as
  \[
     l_{kj}\ : \ (\overline{z_k}-\overline{z_j})z-(z_k-z_j)\overline{z}
         +z_k\overline{z_j}-\overline{z_k}z_j=0.
  \]
  From $ z_k=\varphi_t(w_k) $ and $ w_k\overline{w_k}=1 $, $ l_{kj} $ 
  is written as
  \[
     l_{kj}\ : \ (t^2-w_kw_j)z+(w_kw_jt^2-1)\overline{z}+(1-t^2)(w_k+w_j)=0.
  \]
  Let $ L=l_{12}l_{13}l_{23} $.

  We can eliminate $ w_1,w_2,w_3 $ from $ L=0 $ using \eqref{eq:Bw}.
  (We used Risa/Asir\footnote{{\tt http://www.math.kobe-u.ac.jp/Asir/}
  (Kobe distribution)} Symbolic computation system,
   to compute the Gr\"obner basis and the elimination ideals.)
  We have,
  \[
     L_{\lambda}=\Lambda_2\lambda^2+\Lambda_1\lambda+\Lambda_0=0,
  \]
  where
  \begin{align*}
   \Lambda_2 &=
     (z-t^2\overline{z}+\overline{b}t^2-\overline{b})
     (z-t^2\overline{z}+\overline{a}t^2-\overline{a}) 
     \big((\overline{b}\overline{a}t^2-1)z
        +(t^2-\overline{b}\overline{a})\overline{z}
        +(\overline{a}+\overline{b})(1-t^2)\big),\\
  \Lambda_1 &=
     \big(-(\overline{a}+\overline{b})t^4+(a+b)t^2\big)z^3 
     +\Big(\big(t^2(t^4+2)(\overline{a}+\overline{b})
              -(2t^4+1)(a+b)\big)\overline{z} \\
  & \
     -(t^2-1)\big(2\overline{b}\overline{a}t^4
            -( \vert ab \vert^2+ \vert a+b \vert^2+3)t^2+2ba\big)\Big)z^2  \\
  & \
     +\Big(\big(t^2(t^4+2)(a+b)-(2t^4+1)(\overline{a}+\overline{b})\big)
             \overline{z}^2 \\
  & \ 
     -(t^2-1)\big(( \vert ab \vert^2+ \vert a+b \vert^2+3)(t^4+1)
            -4(ba+\overline{b}\overline{a})t^2\big)\overline{z}\\
  & \ 
     +(t^2-1)^2\big((a+b)(2\overline{b}\overline{a}t^2-1)
         +(\overline{a}+\overline{b})(t^2-2ab)\big)\Big)z\\
  & \
    -\big((a+b)t^4-(\overline{a}+\overline{b})t^2\big)\overline{z}^3 
      -(t^2-1)\big(2bat^4-( \vert ab \vert^2+ \vert a+b \vert^2+3)t^2
      +2\overline{b}\overline{a}\big)
        \overline{z}^2\\
 & \ 
     +(t^2-1)^2\big((\overline{a}+\overline{b})(2abt^2-1)
       +(a+b)(t^2-2\overline{a}\overline{b})\big)\overline{z}
     -(t^2-1)^3\big( \vert ab \vert^2+ \vert a
        +b \vert^2-1\big),\\
  \Lambda_0 &=
   (t^2z-\overline{z}-bt^2+b)(t^2z-\overline{z}-at^2+a) 
   ((t^2-ba)z+(bat^2-1)\overline{z}+(a+b)(1-t^2)).
  \end{align*}

  Then, the envelope of the family of lines 
  $ \{L_{\lambda}\}_{\lambda\in\mathbb{D}} $ is obtained by
  \[
      L_{\lambda}=\frac{\partial}{\partial \lambda} L_{\lambda}=0 
  \]
  (see, for example, \cite[Chap. II]{HW} for envelopes).
  Since $ \frac{\partial}{\partial \lambda}L_{\lambda}=0 $ is a linear equation
  with variable $ \lambda $, it can be solved for $ \lambda $.
  Substituting this solution into $ L_{\lambda}=0 $, we can eliminate $ \lambda $.
  Then, we have 
  \[
  -\frac14
   \frac{\big(t^2z^2-(1+t^4)z\overline{z}+t^2\overline{z}^2+(1-t^2)^2\big)^2
          g_I^{\varphi_t}(z)}{\mbox{Denominator}}=0,
 \]
  where
  \begin{align*}
    \mbox{Denominator}=& \big(z-t^2\overline{z}+\overline{a}(t^2-1)\big)
        \big(z-t^2\overline{z}+\overline{b}(t^2-1)\big)\\
    & \times
     \big((\overline{a}\overline{b}t^2-1)z
      +(t^2-\overline{a}\overline{b})\overline{z}
      -(\overline{a}+\overline{b})(t^2-1)\big).
  \end{align*}
  If the factor $ t^2z^2-(1+t^4)z\overline{z}+t^2\overline{z}^2+(1-t^2)^2 $
  is zero, then it coincides with equation \eqref{eq:E2}
  and represents an ellipse.
  But this is the locus of the vertices of the triangle and does not
  form part of the envelope.
  Hence, 
   \begin{equation}\label{eq:gtildeI}
     {g}_I^{\varphi_t}(z)=\overline{U}z^2+Pz\overline{z}+U\overline{z}^2
                        +\overline{V}z+V\overline{z}+Q=0,
   \end{equation}
   gives the equation of the envelope, where
   \begin{align*}
     U &=(a-b)^2t^4
         +2(2 \vert ab \vert^2- \vert a+b \vert^2+2)t^2
         +(\overline{a}-\overline{b})^2 ,\\
     P &= -2\Big( (2 \vert ab \vert^2- \vert a+b \vert^2+2)(t^4+1) 
              +\big((a-b)^2+(\overline{a}-\overline{b})^2\big)t^2\Big) ,\\
     V &= -2(1-t^2)\Big(\big(( \vert ab \vert^2+1)(a+b)
             -(a^2+b^2)(\overline{a}+\overline{b}) \big)t^2 \\
       & \qquad
         + (a+b)(\overline{a}^2+\overline{b}^2)
          -( \vert ab \vert^2+1)(\overline{a}+\overline{b})\Big), \\
     Q &=(1-t^2)^2\big(( \vert ab \vert^2- \vert a+b \vert^2-1)^2
         -4 \vert a+b \vert^2 \big).
   \end{align*}
  This equation \eqref{eq:gtildeI} gives a conic.

  Now, we need to check that equation
  \eqref{eq:gtildeI} represents a non-degenerate ellipse,
  i.e., we have to check
  \[
      P^2-4U\overline{U} > 0 \quad\mbox{and}\quad
      P(-U\overline{V}^2+PV\overline{V}+4QU\overline{U}-\overline{U}V^2-QP^2)>0
  \] 
  hold.
  See \cite[Lemma 3]{fuji-circum}, for example, 
  to find out the shape of a conic in general form.

  In fact, we have
  \[ 
      P^2-4U\overline{U}=16(1- \vert a \vert^2)(1-\vert b \vert^2)
       \vert 1-a\overline{b} \vert^2(t^4-1)^2 >0 
  \]
  and
  \begin{align*}
    & P(-U\overline{V}^2+PV\overline{V}+4QU\overline{U}-\overline{U}V^2-QP^2)\\
    &\quad
      =32(1- \vert a \vert^2)^2(1- \vert b \vert^2)^2 \vert 1
       -a\overline{b} \vert^4(t^4-1)^2(t^2-1)^2 \\
    & \qquad 
      \times
     \Big(\big \vert (a-b)t^2+\overline{a}-\overline{b}\big \vert^2
      +2(1- \vert a \vert^2)(1- \vert b \vert^2)(t^4+1)\Big) >0.
  \end{align*}
  Therefore, \eqref{eq:gtildeI} represents a non-degenerate ellipse,
  and we have the assertion.
\end{proof}

From Lemma \ref{lemma:GtoS}, if we write the ellipse \eqref{eq:gtildeI}
in the standard form
 \[
       \vert z-f_1 \vert + \vert z-f_2 \vert =r, 
 \]
the foci $ f_1 $ and  $f_2 $ are determined as solutions of 
the following equation
\begin{align}\notag
 & (t^2+1)^2\zeta^2-(t^2+1)\big((a+b)t^2+(\overline{a}+\overline{b})\big)\zeta
  +(\overline{a}+t^2a)(\overline{b}+t^2b)
           -t^2(1-a\overline{a})(1-b\overline{b})=0.
         \label{eq:focif1f2}
\end{align}
The constant $ r $ is obtained from the following
\begin{equation}\label{eq:r2}
  r=\frac12  \vert f_1-f_2 \vert\sqrt{2+\Big\vert\frac{P}{U}\Big\vert} .
\end{equation}
We can express $ r $ as a solution to the following equation by foci
\begin{align} \notag
 R(r^2)= & t^4r^4
   +\big((1+t^4)t^2(f_1f_2+\overline{f_1}\overline{f_2}+2)
    -2(\overline{f_1}f_1+\overline{f_2}f_2)t^4-(1+t^4)^2\big)r^2\\
     \label{eq:R}
  & \quad
     +\big \vert (\overline{f_1}f_2-1)t^4-(f_2^2+\overline{f_1}^2-2)t^2
        +\overline{f_1}f_2-1\big \vert^2=0.
\end{align}
The above \eqref{eq:R} is equivalent to the equality
obtained by Cayley's criterion (see, for example, \cite{GH}). 
See Appendix for details.

\begin{remark}
Equation \eqref{eq:R}  always has two positive solutions,
as we see from the following argument.
Equation \eqref{eq:R} is a quadratic equation with real coefficients
in the variable $ r^2 $,
whose leading coefficient and constant term are positive.
Moreover, the coefficient of the term of degree $ 1 $ are negative,
as follows.

Substituting $ f_1=\rho_1e^{i\theta} $ and $ f_2=\rho_2e^{i\phi} $ 
for the coefficient of degree $ 1 $ and setting $ e^{i(\theta+\phi)}=e^{i\psi} $,
then from $ e^{i\psi}+e^{-i\psi}\leq 2 $, the following inequality holds. 
\begin{align}\notag
 &-\Big(t^8-\big(\rho_1\rho_2(e^{i\psi}+e^{-i\psi})+2\big)t^6 
   +2(\rho_1^2+\rho_2^2+1)t^4-(\rho_1\rho_2(e^{i\psi}+e^{-i\psi})+2)t^2+1\Big)\\
    \label{eq:fu}
 & \qquad \leq  -\Big(t^8-(2\rho_1\rho_2+2)t^6
             +2(\rho_1^2+\rho_2^2+1)t^4-(2\rho_1\rho_2+2)t^2+1\Big).
\end{align}
The right-hand side of the above equation \eqref{eq:fu}
is a quadratic equation with respect to $ \rho_1 $.
Since $ 0<\rho_1,\rho_2 <1 $,
the discriminant $ D $ satisfies
\begin{align*}
   D/4 &= t^4(1-t^2)^2\big((\rho_2^2-2)t^4+2\rho_2^2t^2+\rho_2^2-2\big)  
        = -t^4(1-t^2)^2\big(2(1+t^4)-\rho_2^2(1+t^2)^2 \big)\\
       &< -t^4(1-t^2)^2\big(2(1+t^4)-(1+t^2)^2 \big)
        = -t^4(1-t^2)^2(1-t)^2(1+t)^2<0.
\end{align*}
Therefore, the value of \eqref{eq:fu} is always negative,
and the coefficient of the term of degree $ 1 $ is also always negative.

Hence, we can see that equation \eqref{eq:R} has two positive real roots 
from the following three facts:
1. equation \eqref{eq:R} has at least one real root,
2. $ R(0)>0 $ holds,
and 
3. the axis of symmetry of parabola $ s=R(r^2) $ is in the positive part.
But, the following Lemma \ref{prop:1} shows the
just one root of \eqref{eq:R} gives the $3$-inscribed ellipse $ E_t $.
\end{remark}

The following Lemma is an extension of the result of Frantz
\cite[Proposition 3]{frantz} from the unit disk to a convex domain.
The method of proof is basically the same as that used in 
\cite[Proposition 3]{frantz}.

\begin{lemma}
 \label{prop:1}
Let $ \mathbb{E} $ be an elliptic disk.
For $ f_1,f_2\in\mathbb{E} $,
if there exists a $ 3 $-inscribed ellipse in $ \partial\mathbb{E} $
with foci $ f_1 $ and $ f_2 $, then it is uniquely determined.
\end{lemma}

\begin{proof}
Let $ \widetilde{E} $ be a $ 3 $-inscribed ellipse in $ \partial\mathbb{E} $,
and $ \widetilde{\mathbb{E}} $ the bounded elliptic domain
with boundary $ \widetilde{E} $.

Let $ \vert z-f_1 \vert+\vert z-f_2 \vert =r $
be the standard form of $ \widetilde{E} $.
Assume that there exist three points $ z_1,z_2,z_3 $ 
on the ellipse $ \partial\mathbb{E} $  such that
the triangle $ \triangle(z_1,z_2,z_3) $ is
circumscribed about $ \widetilde{E} $.
From this assumption, $ \widetilde{E} $ is inscribed in 
the angular region $ A $ defined by the angle
$ \angle(z_3,z_1,z_2) $.

For $ r' $ with $ r'<r $, consider the ellipse $ \widetilde{E}' $
defined by $  \vert z-f_1 \vert + \vert z-f_2 \vert =r' $.
Clearly,
$ \widetilde{\mathbb{E}}'\subset \widetilde{\mathbb{E}} $.
Suppose $ \widetilde{E}' $ is inscribed in the angular region
defined by the angle $ A'=\angle(z_3',z_1,z_2') $
with $ z_1 $ as the vertex.
Then, $ A'\subset A $ holds.

The two points $ z_2' $ and $ z_3' $ can be chosen as
points on  $ \partial{\mathbb{E}} $, 
but $ \mathbb{E} $ is a convex set,
so the segment $ [z_2',z_3'] $ is outside the
triangle $ \triangle(z_1,z_2,z_3) $.
Hence, the line segment $ [z_2',z_3'] $ can never be
tangent to $ \widetilde{E}' $ (see, Figure \ref{pic:setu}).

\begin{figure}[htbp]
\centerline{\includegraphics[width=0.4\linewidth]{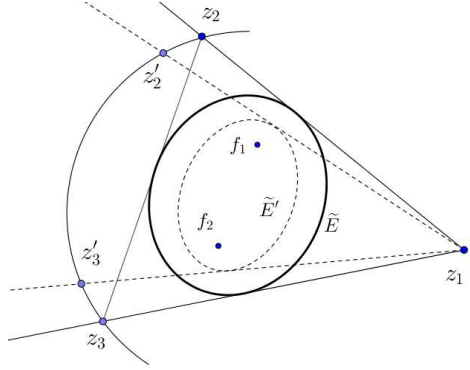}}
\caption{The case of $ \widetilde{\mathbb{E}}'\subset\widetilde{\mathbb{E}}$.}
\label{pic:setu}
\end{figure}
In the case of $ r'>r $, we can show the assertion similarly.
\end{proof}

\begin{remark}\label{rem:contraction}
  As we mentioned before, the Joukowski transformation $ z=\varphi_t(w) $ 
  is a continuous on $ \partial\mathbb{D} $.
  We note that the interior curve of $ B_{\varphi_t} $ is determined only
  by values on $ E_t $, just as the interior curve of a Blaschke product
  $ B $ is determined only by values on the unit circle.
  On the unit circle, $ \varphi_t $ is determined by \eqref{map:phi}.
  Therefore $ \varphi_t $ maps each point on the unit circle
  to a point that is multiplied by $ (t^2-1)/(1+t^2) $
  in the vertical direction,
  \[
    \begin{array}{lccc}
      \varphi_t\vert_{\partial\mathbb{D}} \,:\, 
                 & \partial\mathbb{D} & \to & E_t\\
                 &   \rin              &             &\rin\\
                 &  X+iY              &\mapsto & X+i(t^2-1)/(1+t^2)Y  \\
                 &  w           & \mapsto & \dfrac{t^2w+\overline{w}}{1+t^2}.\\
    \end{array}
  \]
  Note that $  \varphi_t\vert_{\partial\mathbb{D}} $ determines
  a linear transformation of  $ \mathbb{R}^2 $.
  This linearity is also inherited in the correspondence between
  the triangle inscribed in the unit circle 
  and the triangle inscribed  in $ E_t $.
  The same is true for the interior curves.

  In fact, the correspondence of the interior curves is as follows.

  From Theorem \ref{thm:1}, the interior curve associated with
  $ B(z)=z(z-a)(z-b)/((1-\overline{a}z)(1-\overline{b}z)) $ is
  given by 
  \[
      \vert w-a \vert + \vert w-b \vert = \vert 1-\overline{a}b \vert ,
  \]
  which can also be represented as
  \begin{align} \notag
     g_I(w)&=
      (\overline{a}-\overline{b})^2w^2+(a-b)^2\overline{w}^2
      -2\big(2(1+ \vert ab \vert^2)- \vert a+b \vert^2\big)
      w\overline{w} \\ \notag
    & +2\big((1+ \vert ab \vert^2)(\overline{a}+\overline{b})
      -(\overline{a}^2+\overline{b}^2)(a+b)\big)w \\ \notag
    &
       +2\big((1+ \vert ab \vert^2)(a+b)
       -(a^2+b^2)(\overline{a}+\overline{b})\big)\overline{w} \\ \label{eq:gI}
    &+(1- \vert ab \vert )^2
      - \vert a+b \vert^2\big(2(1+ \vert ab \vert^2)
      - \vert a+b \vert^2\big)=0
  \end{align}
  in the general form.
  The ellipse determined by equation  \eqref{eq:gtildeI}
  coincides with the ellipse determined from \eqref{eq:gI}
  multiplied by $ (t^2-1)/(1+t^2) $ in the
  vertical direction (in the direction of the imaginary axis) only.
  Thus, $ g_I^{\varphi_t}(z)=0 $ is equivalent to
  \[
       g_I\Big(\frac{\overline{z}-t^2{z}}{1-t^2}\Big)=0.
  \]
  In fact, 
  $(t^2-1)^2g_I\big((\overline{z}-t^2{z})/({1-t^2})\big)={g}_I^{\varphi_t}(z) $
  holds.
\end{remark}

%

\begin{proof}[Proof of Corollary \ref{cor:unique}]
  For $ C_2 $ a $3$-inscribed ellipse in $ E_t $,
  let $ C'_2 $ be the ellipse that is multiplied by
  $ (1+t^2)/(t^2-1) $  in the vertical direction.
  From Theorem \ref{thm:Fra}, there exists a Blaschke product $ B $
  having $ C'_2 $ as its interior curve.
  Then, the interior curve of Blaschke-like map 
  associated with $ B $ and $ \varphi_t $ coincides with $ C_2 $.

  The converse follows from Theorem \ref{thm:daen}.
\end{proof}

R. Schwartz and S. Tabachnikov \cite{taba} studied the loci of centers of 
mass for one-parameter families of Poncelet $ n $-polygons.
In particular, they show that the locus of centers of mass for 
the vertices of each Poncelet $ n $-polygon is an ellipse 
similar to the outer ellipse.

For the polygons created by the inverse images of a Blaschke-like map,
the locus of centers of mass for the vertices forms an ellipse similar to
the outer ellipse, even if the interior curve does not contain
any conics.

\begin{proof}[Proof of Proposition \ref{thm:jushin}]
Let $ B $ be a canonical Blaschke product
$ B(w)=w\prod_{k=1}^{d-1}(w-a_k)/(1-\overline{a}_kw) $ of degree $ d $,
and $ M(B) $  the locus of centers of mass for the vertices of 
each polygons given by 
$ B^{-1}(\lambda) $ with $ \lambda\in\partial\mathbb{D} $.

First, we show that if $ B $ was a canonical Blaschke product then
$ M(B) $ forms a circle.

For $ \lambda\in\partial\mathbb{D} $, 
let $ w_1,\cdots,w_d\in\partial\mathbb{D} $ be the $ d $
distinct points mapped to $ \lambda $ under $ B $, where they are assumed to
satisfy $ 0\leq \arg w_1,<\arg w_2<\cdots<\arg w_d<2\pi $.
As the points $ w_1,\cdots,w_d $ are the solution of $ B(w)=\lambda $,
we have
\begin{align*}
 (w-w_1)(w-w_2)\cdots(w-w_d)
  &=w(w-a_1)\cdots(w-a_{d-1})
   -\lambda(1-\overline{a}_1w)\cdots(1-\overline{a}_{d_1}w)\\
  &= w^d-\big((a_1+\cdots+a_{d-1})
      +(-1)^{d-1}\overline{a}_1\cdots\overline{a}_{d-1}\lambda\big)w^{d-1}+\cdots\\
  & \qquad
      +\big((-1)^{d-1}a_1\cdots a_d+(\overline{a}_1+\cdots+\overline{a}_{d-1})
       \lambda\big) w-\lambda.
\end{align*}
Set $ \zeta=(w_1+\cdots+w_d)/d $.
Comparing each term of degree $ d-1, 1 $ and $ 0 $ for $ w $
in the above equation, we have
\begin{align} \label{coef:d-1}
 & \zeta d = a_1+\cdots+a_{d-1}
            +(-1)^{d-1}\overline{a}_1\cdots\overline{a}_{d-1}\lambda, \\
              \label{coef:1}
 & (-1)^{d-1}w_1\cdots w_d\big(\frac{1}{w_1}+\!\cdots\!+\frac{1}{w_d}\big)
         = (-1)^{d-1}a_1\cdots a_{d-1}
            +(\overline{a}_1+\!\cdots\!+\overline{a}_{d-1})\lambda,\\
            \label{coef:0}
 & (-1)^{d-1}w_1\cdots w_d =\lambda.
\end{align}
From \eqref{coef:1} and \eqref{coef:0},
\[
    \lambda(\overline{w_1}+\cdots+\overline{w_d})
   =(-1)^{d-1}a_1\cdots a_{d-1}
    +(\overline{a}_1+\cdots+\overline{a}_{d-1})\lambda
\]
holds.
Dividing both sides of the above equation by $ \lambda $, 
and substituting \eqref{coef:d-1}, we have
\begin{equation}\label{eq:bar}
  \overline{\zeta}d=(-1)^{d-1}a_1\cdots a_{d-1}\overline{\lambda}
    +(\overline{a}_1+\cdots+\overline{a}_{d-1}).
\end{equation}
From $  \vert \lambda \vert^2=\lambda\overline{\lambda}=1 $,
eliminating $ \lambda $ from \eqref{coef:d-1} and \eqref{eq:bar}, we have
\[
    \Big \vert \frac{\zeta d-(a_1+\cdots+a_{d-1})}
              {\overline{a}_1\cdots\overline{a}_{d-1}}\Big \vert =1. 
\]
Therefore, $ \zeta $ satisfies
\begin{equation}\label{eq:en}
    \Big \vert \zeta-\frac{a_1+\cdots+ a_{d-1}}{d}\Big \vert 
         =\Big \vert \frac{a_1\cdots a_{d-1}}{d}\Big \vert ,
\end{equation}
and $ \zeta $ is on the circle with center $ (a_1+\cdots+a_{d-1})/d $
and radius $\vert a_1\cdots a_{d-1}\vert/d $.
Note that if the Blaschke product $ B $ has multiple zero at $ z=0 $,
the barycenter degenerate to a single point.

Next, we consider the locus $ M(B_{\varphi_t}) $
for the Blaschke-like map with respect to $ B $ and $ \varphi_t $.
For the same reasons discussed in Remark \ref{rem:contraction},
the locus $ M(B_{\varphi_t}) $ is obtained by 
contracting the above circle \eqref{eq:en} in the horizontal direction.

Since the outer unit circle is also contracted horizontally 
in the same ratio and transforms into an ellipse,
these two ellipses are similar.
\end{proof}

\subsubsection{The exterior curve of $ B_{\varphi_t} $}\

Let $ B $ be a canonical Blaschke product of degree $ d $.
For $ \lambda\in\partial \mathbb{D} $, let $ \mathcal{L}_{\lambda} $
be the set of $ d $ lines tangent to $ \partial\mathbb{D} $
at the $ d $  preimages of $ \lambda $ under $ B $.
Then the trace of the intersection points of each two elements in
$ \mathcal{L}_{\lambda} $ as $ \lambda $ ranges over the unit circle is called
the {\it exterior curve associated with} $ B $.

In \cite {fuji-circum}, we obtained the following result.

\begin{lemma}[{\cite[Theorem 2]{fuji-circum}}]\label{thm:d-1}
  For a canonical Blaschke product of degree $ d $,
  the exterior curve is an algebraic curve of degree at most $ d-1 $.
\end{lemma}

For a Blaschke-like map, we have the following result.

\begin{proposition}\label{prop:d-1}
  For a Blaschke-like map associated with a canonical Blaschke product
  $ B $ of degree $ d $ and the map $ \varphi_t $,
  the exterior curve is an algebraic curve of degree at most $ d-1 $.
\end{proposition}

\begin{proof}
Since the equations of lines tangent to the unit circle at points
$ \omega_1 $ and $ \omega_2 $ are
\[
  \frac{1}{\omega_k}w+\omega_k\overline{w}-2=0\quad (k=1,2),
\]
the intersection point $ w_o $ of these two lines is given by
\begin{equation}\label{eq:ww}
  w_0=\frac{2\omega_1\omega_2}{\omega_1+\omega_2}.
\end{equation}
Note that equality \eqref{eq:ww} is equivalent to
$ \overline{w_0}=2/(\omega_1+\omega_2) $.

The equation of the line tangent to a conic $ C $
  \[
    C\,:\  \overline{u}z^2+pz\overline{z}+u\overline{z}^2
          +\overline{v}z+v\overline{z}+q=0
  \]
at point $ z_0\in C $ can be written as
  \[
      (2\overline{u}z_0+p\overline{z_0}+\overline{v})z
           +(2u\overline{z_0}+pz_0+v)\overline{z}
               +\overline{v}z_0+v\overline{z_0}+2q=0.
  \]
Therefore, the equation of the lines tangent to the ellipse $ E_t $
at points $ \zeta_k=\varphi_t(\omega_k) \ (k=1,2) $ are
\[
   (\omega_k^2-t^2)z-(t^2\omega_k^2-1)\overline{z}+2(t^2-1)\omega_k=0.
\]
The intersection point $ z $ of these two lines is given by 
\begin{equation}\label{eq:zz}
  z=\frac{2(\omega_1\omega_2t^2+1)}{(t^2+1)(\omega_1+\omega_2)}
   =\dfrac1{1+t^2}\Big(t^2w_0+\overline{w_0}\Big),
\end{equation}
where the second equality is obtained by \eqref{eq:ww}.
Relation \eqref{eq:zz} gives the point $ z $ on the exterior curve of
$ B_{\varphi_t} $ corresponding to each point $ w_0 $ on the exterior curve
of $ B $.
Hence, the assertion follows from Lemma \ref{thm:d-1}.
\end{proof}

\section{Outside of parabolic regions}
\subsection{General form and standard form}\label{sec:parab}\

A conformal map
\[
  z=\psi_t(w)=\Big(\frac{1-w}{1+w}+t\Big)^2-t^2 \quad (t>0)
\]
maps conformally the unit disk in the $ w $-plane onto 
the domain $ \mathbb{P}_t $ 
in the $ z $-plane.
Then, the boundary $ P_t=\partial\mathbb{P}_t $ is the parabola whose
focus is $ -t^2 $, and the directrix is given by $ w=t^2 $.
The parabola $ P_t $ is written as
\[
  \dfrac{\vert z+\overline{z}-2t^2 \vert }{2} 
    = \vert z+t^2 \vert , \mbox{\quad or\quad }
  (z-\overline{z})^2=8t^2(z+\overline{z}).
\]
Since $ \psi_t\big((1+t+tw)/(w-t-tw)\big)=\psi_t(w) $ holds,
the map $ \psi_t $ conformally maps
$ \{\vert 1+t(1+w)\vert < \vert w-t(1+w)\vert\} $ onto 
$ \widehat{\mathbb{C}}\setminus\overline{\mathbb{P}}_t $ as well as
conformally maps $ \mathbb{D} $ onto 
$ \widehat{\mathbb{C}}\setminus\overline{\mathbb{P}}_t $.
Note that as the inequality
$ \vert 1+t(1+w)\vert < \vert w-t(1+w)\vert $ is written as
\[
   \begin{cases}
   \left\vert w-\dfrac{2t}{1-2t}\right\vert >\dfrac{1}{1-2t}, 
                      & \mbox{if} \ 0<t<\frac12,\\
   \mbox{Re}\, w < -1 & \mbox{if} \ t=\dfrac12,\\
   \left\vert w-\dfrac{2t}{1-2t}\right\vert <\dfrac{1}{1-2t} 
                      & \mbox{otherwise},
   \end{cases}
\]
so for each $ t>0 $ the domain defined by this inequality 
never intersects $ \mathbb{D} $.
Therefore, for each 
$ z\in\widehat{\mathbb{C}}\setminus\overline{\mathbb{P}}_t $
we can choose a unique branch $ \psi_t^{-1} $ that satisfies $ \vert w\vert<1 $.
Then, the map $ B_{\psi_t} $ maps 
$ \widehat{\mathbb{C}}\setminus\overline{\mathbb{P}}_t $ onto itself,
and is a  Blaschke-like map associated with $ B $ and $ \psi_t $
(see Figure \ref{pic:parab-all}).

\begin{figure}[htbp]
  \centering
  \includegraphics[width=0.7\linewidth]{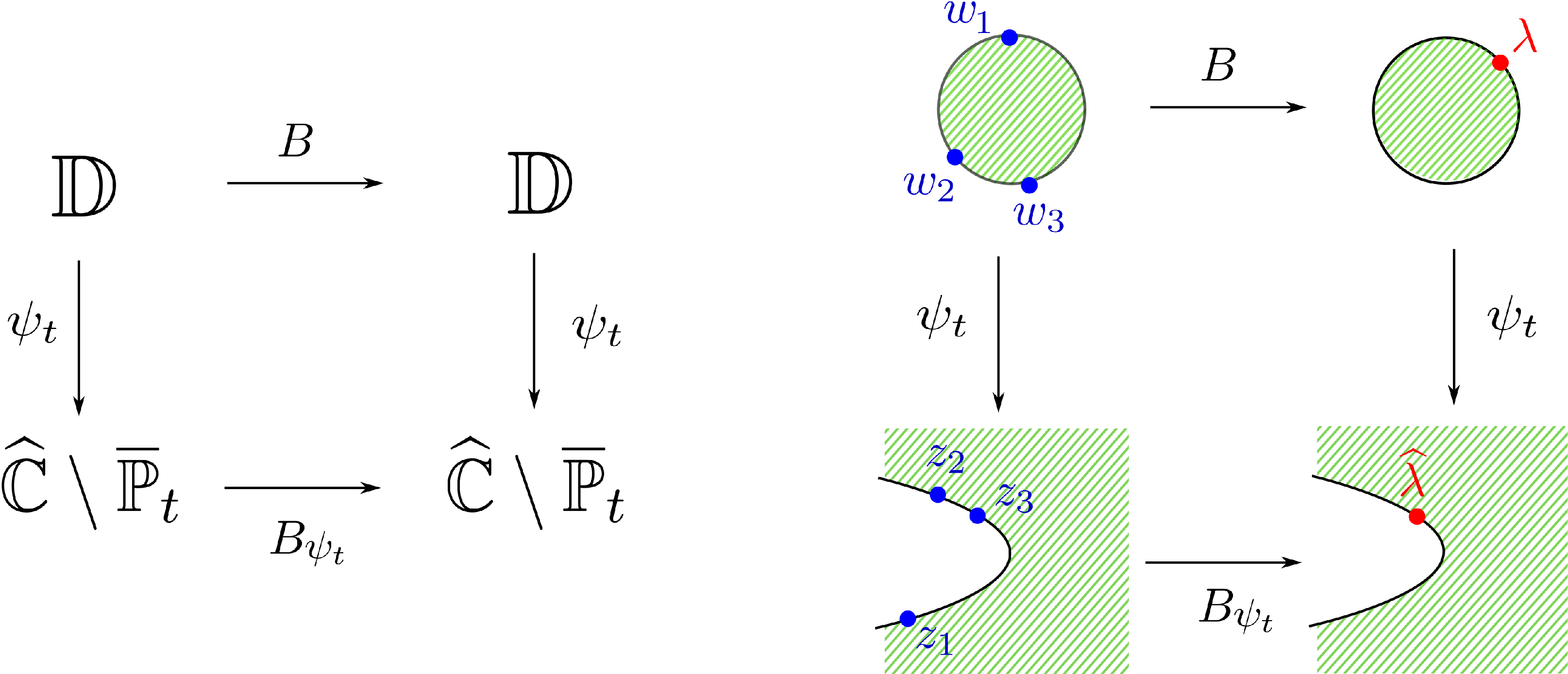}
  \caption{The Blaschke-like map $ B_{\psi_t}=\psi_t\circ B\circ\psi_t^{-1}$.}
  \label{pic:parab-all}
\end{figure}

\subsection{The interior curve of $ B_{\psi_t} $}\

As we described in Section \ref{sec:DGM},
let $ B $ be a canonical Blaschke product of degree $ 3 $
\[
   B(w)=w\frac{w-a}{1-\overline{a}w}\frac{w-b}{1-\overline{b}w} 
   \quad (a,b\in\mathbb{D}), 
\]
and $ B_{\psi_t}=\psi_t\circ B\circ \psi_t^{-1} $ the Blaschke-like map 
associated with $ B $ and $ \psi_t $.

For $ \lambda\in\partial\mathbb{D} $, let $ w_1,w_2,w_3 $ 
be the three points mapped to $ \lambda $ under $ B $, and
set
\[ 
   \widehat{\lambda}=\psi_t(\lambda), \qquad
   z_k=\psi_t(w_k) \quad (k=1,2,3).
\]
Then, $ B_{\psi_t} $ maps each $ z_k $ ($ k=1,2,3$) to $ \widehat{\lambda} $.

Here, we show the interior curve with respect to $ B_{\psi_t} $ forms
an ellipse.
  
\begin{proof}[Proof of Theorem \ref{thm:para}]
This theorem can be proved using the same method as the proof of 
Theorem \ref{thm:daen}.
The $3$-inscribed conic $g_I^{\psi_t}(z)=0 $ is given as follows
\begin{equation}\label{eq:ghat}
   g_I^{\psi_t}(z)=\overline{U}z^2+Pz\overline{z}+U\overline{z}^2
       +\overline{V}z+V\overline{z}+Q=0,
\end{equation}
where
\begin{align*}
 U& =\big(( \vert ab \vert^2- \vert a+b+1 \vert^2)^2-4 \vert a
      +1 \vert^2 \vert b+1 \vert^2\big)t^2 \\
    &\quad
      +2\big( \vert ab \vert^2(a+b-\overline{a}-\overline{b})
      -(\overline{a}+\overline{b}+1)(a^2+b^2+1) 
      +(a+b+1)(\overline{a}^2+\overline{b}^2+1)
       +2(ab-\overline{a}\overline{b})\big)t \\
    & \quad
     +(a-b)^2+(\overline{a}-\overline{b})^2-2 \vert a+b \vert^2
      +4( \vert ab \vert^2+1), \\
 P&=2\big(( \vert ab \vert^2- \vert a+b+1 \vert^2)^2 -4 \vert a+1 \vert^2
          \vert b+1 \vert^2\big)t^2\\
   & \quad
     -2\big(4 \vert ab \vert^2+(a-b)^2+(\overline{a}-\overline{b})^2
     -2 \vert a+b \vert^2+4\big),\\
 V&=-4\big( \vert ab \vert^2( \vert ab \vert^2-2 \vert a+b \vert^2+2)
      +( \vert a+b \vert^2-2)^2
      -(a-b)^2-(\overline{a}-\overline{b})^2+1\big)t^2 \\
    & \quad
     -4\big(( \vert ab \vert^2- \vert a+b-1 \vert^2)
         (a+b-\overline{a}-\overline{b})
      +2(\overline{a}+\overline{b}-2)(ab-1)
     -2(a+b-2)(\overline{a}\overline{b}-1)
     \big)t,\\
 Q&=4\big(( \vert ab \vert^2- \vert a+b-1 \vert^2)^2
      -4 \vert a-1 \vert^2 \vert b-1 \vert^2\big)t^2.
\end{align*}

We can check that the equation $g_I^{\psi_t}(z)=0 $ gives the equation of
an ellipse. In fact,
\begin{align*}
 & P^2-4U\overline{U}=\\
 & \quad
   -64(1- \vert b \vert^2)(1- \vert a \vert^2) \vert 1-a\overline{b} \vert^2t^2
            \big((1- \vert b \vert^2)(1- \vert a \vert^2)
         -4(\mbox{Re}(a)+1)(\mbox{Re}(b)+1)\big)
\end{align*}
holds. Substituting $ a=a_r+a_ii,b=b_r+b_ii $ for the last factor
of the right side of the above equation, we have
\begin{align*}
   &\big(1-(b_r^2+b_i^2)\big)\big(1-(a_r^2+a_i^2)\big)-4(1+a_r)(1+b_r) \\
   &\quad
      < (1-b_r^2)(1-a_r^2)-4(1+a_r)(1+b_r)
      =(1+b_r)(1+a_r)\big((1-b_r)(1-a_r)-4\big) <0.
\end{align*}
Therefore, $ P^2-4U\overline{U} >0 $ holds and 
$g_I^{\psi_t}(z)=0 $ gives an equation of an ellipse or its degeneration.
\end{proof}

\begin{remark}\label{rem:para}
  The map $ z=\psi_t(w) $ is a continuous 
  map on $ \overline{\mathbb{D}} $, and on $ \partial\mathbb{D} $
  it can be written 
  \begin{equation}\label{map:psi}
   z=\psi_t(w)=\frac{(1-2t)w+(1+2t)\overline{w}-2}{w+\overline{w}+2} 
      \qquad (w\in\partial\mathbb{D})
  \end{equation}
  since $ w\overline{w}=1 $ holds.
  Therefore, on $ P_t $, the inverse can be expressed as 
  \begin{equation}\label{map:psiIV}
         w=\psi^{-1}_t(z)=\frac{-(z+\overline{z}+2)t+z-\overline{z}}
                        {(z+\overline{z}-2)t}\qquad (z\in P_t). 
  \end{equation}
  That is, the map defined by \eqref{map:psi}  maps each point 
  $ w=X+iY $ on $ \partial\mathbb{D} $ to
  the point $ z=(X-1)/(X+1)-2iYt/(X+1) $ on $ P_t $,
  \[
    \begin{array}{lccc}
      \psi_t\vert_{\partial\mathbb{D}} \,:\, 
                 & \partial\mathbb{D} & \to & P_t\\
                 &   \rin              &             &\rin\\
                 &  X+iY              &\mapsto & (X-1)/(X+1)-2iYt/(X+1).
    \end{array}
  \]
  The above map is no longer a linear one, but it does not change 
  the degree of each algebraic curve on $ \mathbb{R}^2 $, i.e. on $\mathbb{C}$.
  In particular, it maps a line to a line.

   As in the case of ellipses, 
   the interior curve of $ B_{\psi_t} $ 
   is obtained from the interior curve of $ B $ by the transformation
   $ X+iY \mapsto (X-1)/(X+1)-2iYt/(X+1) $.
   
  Therefore, it follows that the ellipse determined by
  \eqref{eq:ghat} never degenerates to a single point or the empty set,
  since the ellipse that is the interior curve of a Blaschke product
  $ B $ of degree 3 never degenerates.
  In fact, the ellipse determined by \eqref{eq:ghat} and 
  one determined by \eqref{eq:gI} satisfy the equality
  \[
    t^2(z+\overline{z}-2)^2
    g_I\Big( \frac{-(z+\overline{z}+2)t+z-\overline{z}}{(z+\overline{z}-2)t}
          \Big)=g_I^{\psi_t}(z).
  \]

  Note that, $ g_I^{\psi_t} $ can be written as a composition of
  $ g_I $ and a linear transformation of two variables $ x $ and $ y $, 
  where $ z=x+iy $.
\end{remark}

The following result is obtained by the same argument as
Corollary \ref{cor:unique} in Section \ref{sec:DGM}.

\begin{proof}[Proof of Corollary \ref{cor:para}]
  As in the proof of Corollary \ref{cor:unique},
  for $ C_2 $ a $ 3$-inscribed ellipse in $ P_t $,
  let $ C'_2 $ be the ellipse obtained by mapping each point $ z=x+iy $
  on $ C_2 $  by the map $ (x,y)\mapsto \big((1+x)/(1-x),y/(t(1-x))\big) $.
  It now follows that $ C'_2 $ is an ellipse, 
  since the map does not change degree of curves 
  and the image is bounded as $ \{1\}\not\in C_2 $.

  From Theorem \ref{thm:Fra}, there exists a Blaschke product $ B $
  having $ C'_2 $ as its interior curve.
  Then, the interior curve of Blaschke-like map 
  associated with $ B $ and $ \varphi_t $ coincides with $ C_2 $.

  The converse follows from Theorem \ref{thm:para}.
\end{proof}

For the exterior curves, as in Proposition \ref{prop:d-1},
the following holds.

\begin{proposition}
  For a Blaschke-like map associated with a canonical Blaschke product
  $ B $ of degree $ d $ and the map $ \psi_t $,
  the exterior curve is an algebraic curve of degree at most $ d-1 $.
\end{proposition}

\begin{proof}
  From Lemma \ref{thm:d-1},
  for a canonical Blaschke product of degree $ d $,
  the exterior curve is an algebraic curve of degree at most $ d-1 $.
  Let $ g_E(w)=0 $ be a defining equation of the exterior curve 
  associated with $ B $.

  Then, a defining equation $ g_E^{\psi_t}(z)=0 $
  of a Blaschke-like map $ B_{\psi_t} $ is given as
  the condition that the numerator of the composition of 
  $ g_E(w) $ and 
  \begin{equation}\label{eq:1ji}
   w=\frac{-(z+\overline{z}+2)t+z-\overline{z}}{(z+\overline{z}-2)t} 
  \end{equation}
  is zero. 
  Hence, the degree of the equation is not more than $ d-1 $.
\end{proof}

Here, the simple question arises, 
"If the interior or exterior of the unit disk is mapped onto
 a domain whose boundary is a conic by a conformal map, 
 does the same property always hold?" 
In the next section, 
we give an example where the answer to this question is ``no''.


\section{Inside of elliptic disks}\label{sec:inside}\

Using Jacobi elliptic functions,
it is possible to construct a conformal map
that maps the unit disk to the elliptic domain.
See, for example,  Nehari \cite[Sec. VI]{nehari},
Schwarz \cite{schwarz} and Szeg\"o \cite{szego}
for details.

Let $ \mbox{sn}^{-1}(w,k) $ be the elliptic integral of the first kind
\[ 
   \mbox{sn}^{-1}(w,k)=\int_0^w\frac{d\omega}{\sqrt{(1-\omega^2)(1-k^2\omega^2)}},
\]
and let
\[
    K(k)=\int_0^1\frac{d\omega}{\sqrt{(1-\omega^2)(1-k^2\omega^2)}},\quad
    K'(k)=\int_1^{\frac1k}\frac{d\omega}{\sqrt{(\omega^2-1)(1-k^2\omega^2)}}.
\]
Here, for $ 0<p<1 $, we consider the following transformations
\begin{align*}
  & u(w)=u=\frac{w-1}{w+1},\quad  v(u)=v=c \cdot \mbox{sn}^{-1}(u,k),\\
  & x(v)=x=\sqrt{\frac{1+p}{1-p}}e^v,\quad
    z(x)=z=\frac{\sqrt{1-p^2}}2\Big(x+\frac1x\Big),
\end{align*}
where $ k $ and $ c $ are chosen such that 
$ \log\sqrt{(1+p)/(1-p)}=\pi(K(k))/(K'(k)) $ and \\*
$ \log\sqrt{(1+p)/(1-p)}=cK(k) $.

Set
\begin{align*}
  D_w &= \big\{ \vert w \vert <1, \ \mbox{Im} (w)>0 \big\}, \quad
          D_u = \big\{ \mbox{Re} (u)<0,\ \mbox{Im} (u)>0 \big\}, \\
  D_v &= \big\{-\log\sqrt{\frac{1+p}{1-p}}<\mbox{Re} (v)<0,\
               0< \mbox{Im} (v)<\pi \big\}, \\
  D_x &= \big\{1< \vert x \vert <\sqrt{\frac{1+p}{1-p}}, \ 
                 \mbox{Im} (x)>0 \big\}, \\
  D_z &= \big\{\vert z-\sqrt{1-p^2} \vert + \vert z+\sqrt{1-p^2} \vert <2, \
                \mbox{Im} (z)>0 \big\}.
\end{align*} 

Then, $ u,\,v,\,x $, and $ z $ are conformal maps of 
$ D_w $ to $ D_u $,
$ D_u $ to $ D_v $,
$ D_v $ to $ D_x $, and 
$ D_x $ to $ D_z $, respectively (see Figure \ref{pic:Jacobi}). 

\begin{figure}[htbp]
\centerline{\includegraphics[width=0.8\linewidth]{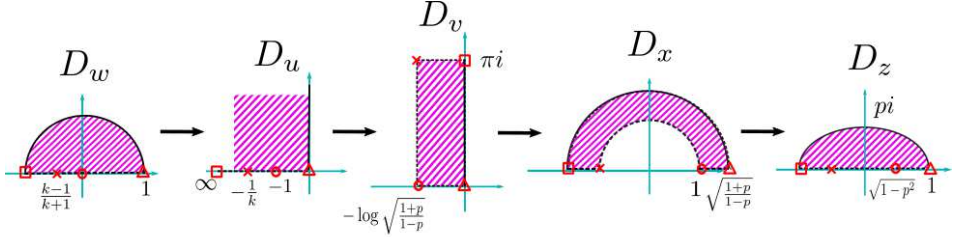}}
\caption{Conformal maps $ u,\,v,\,x$, and $ z $.}
\label{pic:Jacobi}
\end{figure}

Thus, the composition
\[  
   z=\gamma(w)=z\circ x\circ v\circ u (w)
\]
is a conformal map that maps the upper half of the unit disk
onto the upper half of the elliptic domain,
and the diameter $(-1,1) $ of the unit disk 
to the major axis $(-1,1)$ of the elliptic domain.
By Schwarz reflection principle, there exists a conformal map 
$ \widetilde{\gamma} $ from 
the unit disk $ \mathbb{D} $ onto the elliptic domain
$ \mathcal{E}_p=\big\{  \vert z-\sqrt{1-p^2} \vert 
 + \vert z+\sqrt{1-p^2} \vert <2\big\} $.

Then, the map 
$ B_{\widetilde{\gamma}} = \widetilde{\gamma}\circ B\circ \widetilde{\gamma}^{-1} $
is a Blaschke-like map on $ \mathcal{E}_p $.
\renewcommand{\arraystretch}{1}

\subsection{Computer experiments}\

Let $ p=0.800438\cdots$ and
$ B(w)=w(w-a)(w+a)/((1-\overline{a}w)(1+\overline{a}w)) $
with $ a=3/10 $.
Now, we consider a Blaschke-like map $ B_{\widetilde{\gamma}} $ constructed with
$ \widetilde{\gamma} $ obtained in the above procedure.

Figure \ref{pic:1} indicates the family of lines 
connecting the inverse images of points on the ellipse
$ \partial\mathcal{E}_{p} $ with $ p\approx0.800438$,
using Mathematica\footnote{{\tt https://www.wolfram.com/}}.
The envelope gives the interior curve with respect ot $ B_{\widetilde{\gamma}}$,
but it is not an ellipse.

\begin{figure}[htbp]
\centerline{\includegraphics[width=0.4\linewidth]{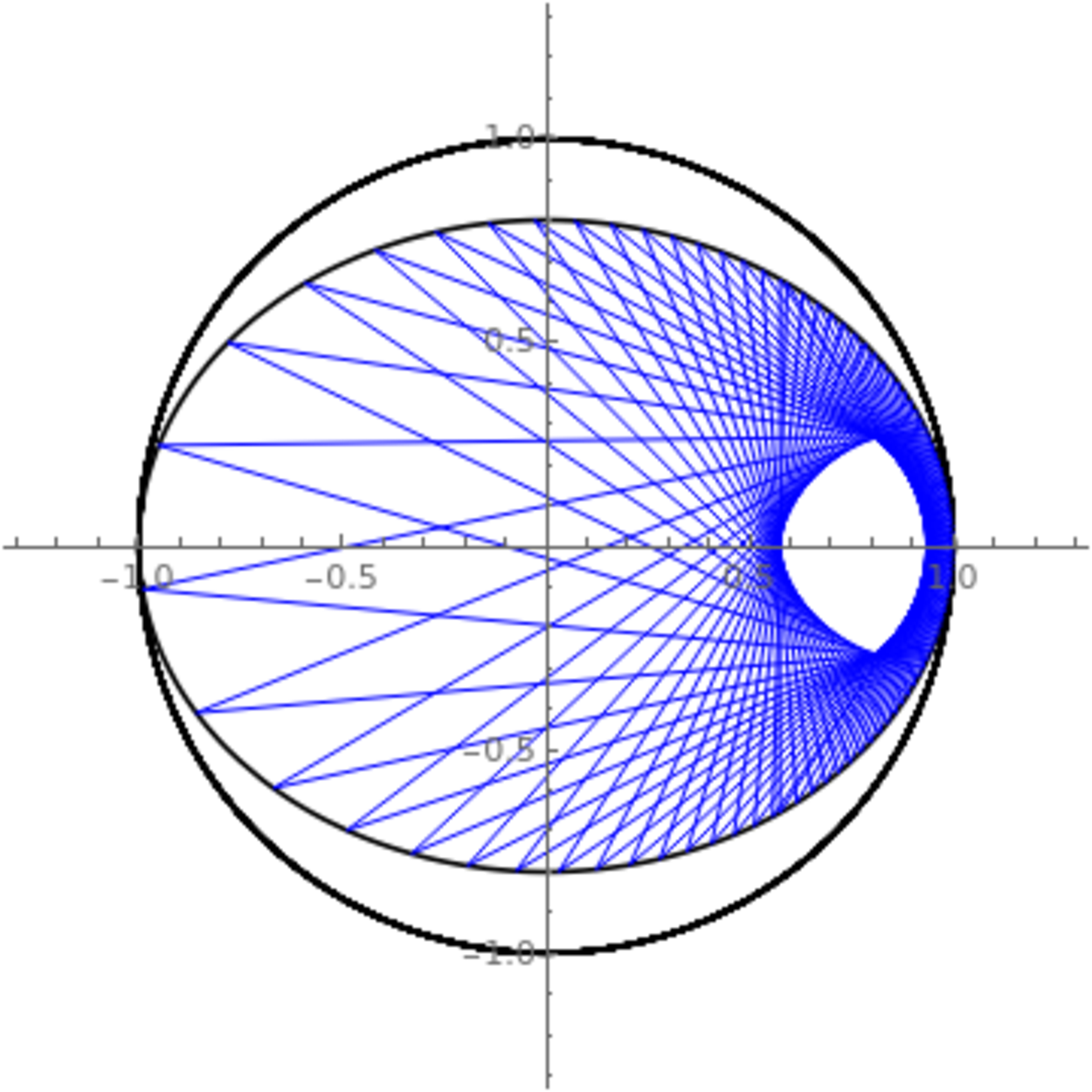}}
\caption{The interior curve with respect to
         Blaschke-like map for $ k=0.045 $ and the canonical Blaschke product
          $ B(w)=w(w^2-a^2)/(1-\overline{a}^2w^2) $ with $ a=3/10 $.}
\label{pic:1}
\end{figure}


\begin{appendix}
\section{The Cayley criterion}

For the two conics $ C $ and $ D $, Cayley  (\cite{cayley1,cayley3}) 
gives the condition so-called the ``Cayley criterion'' that
there exists an $ n $-sided polygon inscribed in $ C $ and 
circumscribed about $ D $.
Later, Griffiths and Harris prove this criterion 
using the algebraic geometrical method \cite{GH}.

Here, we find again the condition that there exists a triangle 
inscribed in the ellipse $ E_t $ and circumscribed about
$ E:  \vert z-f \vert + \vert z-g \vert =r $.

First, consider $ E_t $ as a curve in the projective space 
$ \mathbb{P}^2(\mathbb{R}) $ and put $A$ as the matrix representation.
More specifically, when the equation  $a_{11}x_1^2+2a_{12}x_1x_2+2a_{13}x_1x_3
+a_{22}x_2^2+2a_{23}x_2x_3+a_{33}x_3^2=0 $ is obtained by 
substituting $ z=x_1/x_3+i x_2/x_3 $ into the defining equation of the ellipse,
we can write $ A=(a_{ij}) $. For $ E_t $, we have
\[ 
  A=\begin{pmatrix}
      (t^2-1)^2 & 0 & 0\\ 0 & (t^2+1)^2 & 0 \\ 0 & 0 & -(t^2-1)^2
    \end{pmatrix}.
\]
Next, an ellipse with foci $ f $ and $ g $ is 
similarly represented in matrix form $ B=(b_{ij}) $ as
\begin{align*}
   b_{11} &= -4(r^2-(f_r-g_r)^2),\\  
   b_{12} &=4(f_i-g_i)(f_r-g_r),\\
   b_{13} &=2((f_r+g_r)r^2-(f_r-g_r)(f_r^2+f_i^2-g_r^2-g_i^2),\\
   b_{22} &=-4(r^2-(f_i-g_i)^2),\\
   b_{23} &=2((f_i+g_i)r^2-(f_i-g_i)(f_r^2+f_i^2-g_r^2-g_i^2),\\
   b_{33} &=r^4-2(f_r^2+f_i^2+g_r^2+g_i^2)r^2+(f_r^2+f_i^2-g_r^2-g_i^2)^2,
\end{align*}
where $ f=f_r+i f_i , g=g_r+i g_i $.

From the Cayley criterion, 
a triangle inscribed in $E_t $ and circumscribed about $E$ exists 
if the coefficient $c_2$ of the power series
$ \sqrt{\det(sA+B)}=\sum_{n=0}^{\infty}c_ns^n $
is zero.

Here, for $ F(s)=\det(sA+B) $,
\[ 
   (\sqrt{F})''=\frac{-(F')^2+2FF''}{4F\sqrt{F}}
\]
holds.
Using the above, the coefficient $ c_2=0 $ is equivalent to
$ (-(F')^2+2FF'') \vert _{s=0}=0 $. Eliminating the
non-essential factors, we have
\begin{align*}
  & t^4r^4
    +\big((1+t^4)t^2(fg+\overline{f}\overline{g}+2)
        -2(f\overline{f}+g\overline{g})t^4-(1+t^4)^2\big)r^2 \\
  & \quad
   +\big \vert (\overline{f}g-1)t^4-(g^2+\overline{f}^2-2)t^2
    +\overline{f}g-1\big \vert^2=0.
\end{align*}
Substituting $ f=f_1$ and $ g=f_2 $, the above equation 
is equivalent to equation \eqref{eq:r}.
This equation may have two solutions that satisfy $ r>0 $.
From the result of Lemma \ref{prop:1}, one of them gives
the desired Poncelet $ 3 $-inscribed ellipse.
The other is a solution that does not satisfy the assumption,
for example, the corresponding ellipse intersects the outer ellipse $ E_t $.
\end{appendix}

\bigskip

\textbf{Acknowledgments.}
This work was partially supported by JSPS KAKENHI
Grant Number JP19K03531.
We are grateful to Professor M. Lachance for his valuable comments 
on Section 2.


\end{document}